\documentclass{amsart}

\usepackage[all]{xy}
\usepackage{amsmath,amssymb}
\usepackage{amsthm}
\usepackage{enumerate}
\usepackage[dvips]{graphicx}
\usepackage{txfonts}
\usepackage{url}

\theoremstyle{plain}
  \newtheorem{theorem}{Theorem}[section]
  \newtheorem{corollary}[theorem]{Corollary}

  \newtheorem{proposition}[theorem]{Proposition}
  
  \newtheorem{main}{Main Theorem}

\theoremstyle{definition}
  \newtheorem{definition}[theorem]{Definition}
  
  \newtheorem{example}[theorem]{Example}
  
  \newtheorem{remark}[theorem]{Remark}
    \newtheorem{construction}[theorem]{Construction}

\newcommand{\op}{\mathrm{op}}
\newcommand{\CMS}{\mathbf{CMS}}
\newcommand{\QMS}{\mathbf{QMS}}
\newcommand{\CPS}{\mathbf{CPS}}
\newcommand{\QPS}{\mathbf{QPS}}
\newcommand{\BMS}{\mathbf{BMS}}
\newcommand{\CCMS}{\mathbf{CCMS}}
\newcommand{\BPS}{\mathbf{BPS}}
\newcommand{\CCPS}{\mathbf{CCPS}}

\newcommand{\bR}{\mathbb{R}}
\newcommand{\bC}{\mathbb{C}}

\newcommand{\cF}{\mathcal{F}}
\newcommand{\cB}{\mathcal{B}}
\newcommand{\cN}{\mathcal{N}}

\newcommand{\cC}{\mathcal{C}}
\newcommand{\cK}{\mathcal{K}}
\newcommand{\tr}{\mathrm{tr}}

\def\i<#1>{\langle #1 \rangle}
\def\l<#1>{\left\langle #1 \right\rangle}

\begin{document}

\author{Hitoshi Motoyama and Kohei Tanaka}
\address{College of Economics,
Aoyama Gakuin University, 4-4-25 Shibuya, Shibuya-ku, Tokyo 150-8366, Japan.}
\email{hitoshi@aoyamagakuin.jp}
\address{Institute of Social Sciences, School of Humanities and Social Sciences, Academic Assembly, Shinshu University, 3-3-1 Asahi, Matsumoto, Nagano 390-8621, Japan.}
\email{tanaka@shinshu-u.ac.jp}

\title {Classical and quantum conditional measures from a categorical viewpoint}
\thanks{This work was supported by the Sasakawa Scientific Research Grant and the Nagano Society for the Promotion of Science.
}

\maketitle



\begin{abstract}
This paper presents categorical structures on classical measure spaces and quantum measure spaces in order to deal with canonical maps associated with conditional measures as morphisms. We extend the Riesz-Markov-Kakutani representation theorem and the Gelfand duality theorem to an equivalence of categories between them. From this categorical viewpoint, we introduce a quantum version of conditional measures as a dual concept of the classical one.
\end{abstract}

\section{Introduction}

This paper focuses on category structures on measure (probability) spaces, i.e.,
on morphisms between two given measure spaces.
Several researchers have already introduced notions of morphisms between measure or probability spaces \cite{IH92}, \cite{Gir82}, \cite{Lyn} for various purposes.

A natural approach is to define morphisms as measurable maps that preserve measures. Specifically, a morphism $f : (\Omega, \cF,\mu) \to (\Omega',\cF',\mu')$ between two measure spaces can be defined as a measurable map $f : (\Omega,\cF) \to (\Omega', \cF')$ satisfying $\mu(f^{-1}(A))=\mu'(A)$ for each $A \in \cF'$.
However, the above equality is too strict for the categorical treatment of morphisms. For example, let $B$ be a measurable subspace in a measure space $(\Omega, \cF,\mu)$. It yields the {\em conditional measure space} $(B, \cF_{B}, \mu_B)$ by restricting the original measure space onto $B$. Thus, it is equipped with the canonical inclusion $i : (B,\cF_{B}) \hookrightarrow (\Omega,\cF)$, but it does not preserve measures in general.

This paper aims to extend the class of measure-preserving maps to one containing such inclusions associated with conditional measures. Our approach is based on the notion of bounded liner operators on normed spaces. We introduce the concept of norm for measurable maps and the class of bounded measurable maps. The category $\CMS$ of measure spaces with bounded measurable maps contains canonical inclusions associated with conditional measure spaces as morphisms whose norm is $1$.

On the other hand, quantum probability theory was developed as an algebraic analog of classical probability theory \cite{HO07}, \cite{AO03}. We derive a category structure on quantum measure (probability) spaces with bounded homomorphisms, denoted by $\QMS$, similarly to the case of $\CMS$. A quantum measure space $(A,\varphi)$ consists of a $*$-algebra $A$ and a positive linear map $\varphi$. When $A$ is a commutative $C^*$-algebra, it can be expressed as a classical measure space by the Riesz-Markov-Kakutani (RMK) representation theorem and the Gelfand duality theorem.
This paper extends the two above-mentioned theorems to an equivalence of categories between full subcategories of $\CMS$ and $\QMS$.

\begin{main}[Theorem \ref{main1}]
The category of Borel measure spaces as a full subcategory of $\CMS$ is equivalent to the opposite category of commutative $C^*$-measure spaces as a full subcategory of $\QMS^{\op}$.
\end{main}

From the viewpoint of this duality, we provide a quantum version of conditional measure spaces. The classical conditioning on a measure space is essentially based on choosing subspaces and restricting measures on them. According to the duality, we define quantum conditioning as choosing ideals of an algebra and taking quotients by them.
Given a quantum measure space $(A,\varphi)$ with an ideal $I$, we establish a quantum measure on the quotient algebra $A/I$ using the Gelfand-Naimark-Segal (GNS) construction \cite{KR97}.
We call it the {\em quantum conditional measure} of $(A,\varphi)$ on $A/I$. The following theorem justifies it as a natural quantum analog of a classical conditional measure.

\begin{main}[Theorem \ref{main2}]
Any quantum conditional measure of a commutative $C^*$-measure space is isomorphic to the induced measure from a classical conditional measure in $\QMS_{\sharp}$.
\end{main}

The remainder of this paper is organized as follows.
The first part of Section 2 presents a category structure on classical measure spaces.
We define morphisms on measure spaces as bounded measurable maps, similarly to bounded liner operators on norms spaces. This is advantageous for dealing with conditional measures and describing normalized probabilities for measures in terms of adjoint functors.
The second part of Section 2 is a quantum analog of the first part, based on quantum probability theory.
In addition, it presents typical examples of quantum measure spaces.
The final part of Section 2 examines relationships between the categories of classical and quantum measure spaces. We extend the RMK representation theorem and the Gelfand duality theorem to an equivalence of categories between Borel measure spaces and commutative $C^*$-measure spaces.

Section 3 discusses conditioning in quantum measure spaces.
Classical conditional measures are defined by subspaces and restrictions, whereas quantum conditional measures are defined by ideals and quotients. In the commutative case, quantum conditional measures of $C^*$-measure spaces are essentially derived from classical conditional measures.


\section{Categories of classical and quantum measure spaces}

In this section, we investigate morphisms between measure (probability) spaces.
Several approaches have been adopted in this regard, such as measure-preserving maps \cite{IH92}, maps for statistics \cite{Gir82}, and measurable maps excluding measures \cite{Lyn}.
Here, we introduce another notion. For basic category theory, we refer the readers to Mac Lane's book \cite{Mac98}.

\subsection{Classical measure spaces and their category}

A measurable space $(\Omega, \cF)$ consists of a set $\Omega$ and a $\sigma$-field $\cF$ on $\Omega$. A (classical) {\em measure space} $(\Omega,\cF,\mu)$ consists of a measurable space $(\Omega, \cF)$ with a measure function $\mu : \cF \to \bR_{\geq 0}$. When $\mu(\Omega)=1$, we call it a {\em probability space}. Throughout this paper, we only deal with finite measure spaces.

\begin{definition}\label{classical_bounded}
Given two measure spaces $(\Omega, \cF,\mu)$ and $(\Omega',\cF',\mu')$, a measurable map $f : (\Omega,\cF) \to (\Omega',\cF')$ is {\em bounded} with respect to $\mu$ and $\mu'$ if there exists $M>0$ such that
\[
\mu(f^{-1}(A)) \leq M \mu'(A)
\]
for any $A \in \cF'$. In this case, we define
\[
|f|=\inf \{M>0 \mid \mu(f^{-1}(A)) \leq M \mu'(A), A \in \cF'\}
\]
and call it the {\em norm} of $f$. Furthermore, we say that $f$ is {\em measure-preserving} if
$\mu(f^{-1}(A)) = \mu'(A)$ for any $A \in \cF$. A measure-preserving map is obviously bounded with norm $1$.
\end{definition}

A bounded measurable map is an analog of bounded linear operators on normed spaces.
Let $\CMS$ denote the category of measure spaces and bounded measurable maps,
and let $\CMS_{\sharp}$ denote its subcategory of measure spaces and measure-preserving maps.

\begin{example}
Let $(\Omega,\cF,\mu)$ be a measure space.
The identity map $(\Omega,\cF,\mu) \to (\Omega,\cF,2\mu)$ is an isomorphism in $\CMS$ with norm $1/2$. Further, the inverse map is given by the identity $(\Omega,\cF,2\mu) \to (\Omega,\cF,\mu)$ with norm $2$.
\end{example}

\begin{example}\label{norm_1}
Let $(\Omega,\cF,\mu)$ be a measure space, and let $(B,\cF_{B},\mu_{B})$ be the conditional measure space for a subspace $B \in \cF$.
The inclusion $i : (B,\cF_B,\mu_B) \hookrightarrow (\Omega,\cF,\mu)$ is bounded with norm $1$.
\end{example}

Let $\CPS$ denote the full subcategory of $\CMS$ consisting of probability spaces.
The canonical normalization functor
\[
\cN : \CMS \to \CPS
\]
is given by $\cN(\Omega,\cF,\mu) = (\Omega, \cF, \mu/\mu(\Omega))$.
Note that, for a bounded measure map $f : (\Omega,\cF,\mu) \to (\Omega',\cF',\mu')$, we define $\cN f$ to be $f$ as a map; however, its norm is different from that of $f$:
\[
|\cN f| = \dfrac{\mu'(\Omega')}{\mu(\Omega)}|f|.
\]

\begin{proposition}\label{normalization}
The canonical inclusion functor $\CPS \to \CMS$ is left adjoint to the normalization functor $\cN$.
\end{proposition}
\begin{proof}
For a measure space $(\Omega,\cF,\mu)$ and a probability space $(\Omega',\cF',P)$,
a measurable map $f : (\Omega,\cF) \to (\Omega',\cF')$ with respect to $\mu$ and $P$ is bounded if and only if it is bounded with respect to $\mu/\mu(\Omega)$ and $P$.
Hence, the normalization functor yields a natural isomorphism:
\[
\cN : \CMS((\Omega,\cF,\mu),(\Omega',\cF',P)) \cong \CPS((\Omega,\cF,\mu/\mu(\Omega)),(\Omega',\cF',P)).
\]
\end{proof}

\subsection{Quantum measure spaces and their category}

Quantum probability theory was developed in the 1980s as an algebraic analog of classical probability theory \cite{HO07}, \cite{AO03}.
The commutative case can essentially be regarded as classical probability theory; hence, quantum probability theory is also referred to as non-commutative probability theory.
A quantum measure (probability) space is defined in purely algebraic terms.

\begin{definition}
A $*$-algebra $A$ is a $\bC$-algebra equipped with a $*$-operator. Throughout this paper, assume that a $*$-algebra $A$ has a unit $e \in A$.
Denote the set of positive elements of $A$ by $A_{+}=\{a^*a \mid a \in A\}$.
A $\bC$-homomorphism $\varphi : A \to \bC$ is called a {\em quantum measure} or is said to be {\em positive} if $\varphi(a) \geq 0$ for each $a \in A_{+}$.
A {\em quantum measure space} is a pair $(A,\varphi)$ of a $*$-algebra $A$ and a quantum measure $\varphi : A \to \bC$.
When $\varphi$ preserves the unit, i.e., $\varphi(e)=1$, we call it a {\em state} or an {\em expectation} on $A$.
A {\em quantum probability space} is a quantum measure space $(A,\varphi)$ in which $\varphi$ is a state on $A$.
\end{definition}

The class of quantum measure spaces admits a similar categorical structure to $\CMS$.

\begin{definition}
Given two quantum measure spaces $(A,\varphi)$ and $(A',\varphi')$,
a $*$-algebra homomorphism $f : A \to A'$ is said to be {\em bounded} with respect to
$\varphi$ and $\varphi'$ if there exists $M>0$ such that
\[
\varphi'(f(a)) \leq M \varphi(a)
\]
for any positive element $a \in A_{+}$. In this case, we define
\[
|f|=\inf \{M>0 \mid \varphi'(f(a)) \leq M \varphi(a), a \in A_+\}
\]
and call it the {\em norm} of $f$. Furthermore, $f$ is said to be {\em measure-preserving} if $\varphi'(f(a))=\varphi(a)$ for any $a \in A$. Let $\QMS$, $\QMS_{\sharp}$, and $\QPS$ denote the category of quantum measure spaces with bounded homomorphisms, the subcategory consisting of quantum measure spaces with measure-preserving homomorphisms, and the full subcategory consisting of quantum probability spaces, respectively.
\end{definition}

The quantum version of the normalization functor,
\[
\cN : \QMS \to \QPS,
\]
is given by $\cN(A,\varphi)=(A,\varphi/\varphi(e))$. Here, we use the same notation as that in the classical case. The following proposition can be shown similarly to Proposition \ref{normalization}.

\begin{proposition}
The canonical inclusion functor $\QPS \to \QMS$ is left adjoint to the normalization functor $\cN$.
\end{proposition}

Let us recall some examples of quantum measure spaces.
It is well known that two types of commutative quantum measure spaces are induced from classical measure theory: $W^{*}$-measure spaces and $C^{*}$-measure spaces.
A quantum measure space $(A,\varphi)$ is called a $W^{*}$-measure (resp. $C^{*}$-measure) space when $A$ is a $W^{*}$-algebra (resp. $C^{*}$-algebra).

\begin{example}
Let $(\Omega,\cF,\mu)$ be a measure space.
Let $L^{\infty}(\Omega)$ be the $W^{*}$-algebra of essentially finite measurable functions $f : \Omega \to \bC$. It is equipped with a quantum measure $\varphi_{\mu}$ given by
\[
\varphi_{\mu}(f) = \int_{\Omega} f \ d \mu.
\]
The pair $(L^{\infty}(\Omega),\varphi_{\mu})$ is called the $W^{*}$-measure space associated with $(\Omega,\cF,\mu)$.
\end{example}

The $W^{*}$-algebra $L^{\infty}(\Omega)$ is commutative for a measure space $(\Omega,\cF,\mu)$. For the characteristic function
$\chi_{E}$ of $E \in \cF$ (given by $\chi_{E}(x)=1$ if $x \in E$ and $\chi_{E}(x)=0$ otherwise), we have $\varphi_{\mu}(\chi_{E})=\mu(E)$. Hence, $(L^{\infty}(\Omega),\varphi_{\mu})$ contains a considerable amount of statistical information regarding $(\Omega,\cF,\mu)$.

Another commutative example is $C^{*}$-measure spaces.

\begin{example}
Let $X$ be a compact Hausdorff space and let $\mu$ be a regular Borel measure on $X$.
We denote the $C^{*}$-algebra of continuous $\bC$-valued functions by $C(X)$. It is equipped with a quantum measure $\varphi_{\mu}$ given by
\[
\varphi_{\mu}(f) = \int_{X} f \ d \mu.
\]
The pair $(C(X),\varphi_{\mu})$ is called the $C^{*}$-measure space associated with $(X,\cB(X),\mu)$.
\end{example}

Conversely, for any quantum measure $\varphi$ on $C(X)$, the RMK representation theorem \cite{Rud87} determines a unique regular Borel measure $\mu$ on $X$ such that $\varphi_{\mu}=\varphi$.
Furthermore, for any commutative $C^*$-algebra $A$, the Gelfand duality theorem \cite{GN94} determines a unique compact space $X$ up to isomorphism such that $A \cong C(X)$.
The following fact follows from these two well-known theorems.

\begin{theorem}[Theorem 2.60 of \cite{AO03}]\label{RMK}
For a commutative $C^*$-measure space $(A,\varphi)$, there exists a regular Borel measure $\mu$ on a compact Hausdorff space $X$ such that $(A,\varphi) \cong (C(X),\varphi_{\mu})$ in $\QMS_{\sharp}$.
\end{theorem}

On the other hand, the next two examples are typical non-commutative measure spaces.

\begin{example}
Let $M_{n}(\bC)$ be the $n$-th matrix algebra over $\bC$.
The trace $\tr : M_{n}(\bC) \to \bC$ is positive, and we call it the {\em trace measure} on $M_{n}(\bC)$. The {\em trace state} is the normalization of the trace measure given by $\tr(T)/n$ for $T \in M_{n}(\bC)$.
\end{example}

\begin{example}\label{vector_measure}
Let $B(H)$ be the algebra of bounded linear operators on a Hilbert space $H$.
Fix an object $h \in H$. The {\em vector measure} $w_{h} : B(H) \to \bC$ is defined by $\i<h,\eta h>$ for $\eta \in B(H)$.
When $|h|=1$, the vector measure $w_{h}$ is called the {\em vector state} on $B(H)$.
\end{example}

\subsection{Gelfand duality on categories of measure spaces}

We describe relations between classical and quantum measure spaces in terms of functors.
A measure space $(\Omega,\cF,\mu)$ is associated with a $W^*$-measure space $(L^{\infty}(\Omega), \varphi_{\mu})$, and a measurable map $f : (\Omega,\cF,\mu) \to (\Omega',\cF',\mu')$ induces a $\ast$-homomorphism $L^{\infty}f : L^{\infty}(\Omega') \to L^{\infty}(\Omega)$ by composition with $f$.

\begin{proposition}\label{L}
$L^{\infty} : \CMS \to \QMS^{\op}$ is a functor.
\end{proposition}
\begin{proof}
For a bounded morphism $f : (\Omega,\cF,\mu) \to (\Omega', \cF',\mu')$ in $\CMS$,
it suffices to verify that $L^{\infty}f : (L^{\infty}(\Omega'),\varphi_{\mu'}) \to (L^{\infty}(\Omega),\varphi_{\mu})$ is bounded in $\QMS$.
Note that a positive element in $L^{\infty}(\Omega')$ is a function taking non-negative real values.
Since there exists $M>0$ satisfying $\mu(f^{-1}(A)) \leq M \mu'(A)$ for each $A \in \cF'$, we have the following inequality for any positive function $g$:
\[
\varphi_{\mu}(L^{\infty}f(g))=\varphi_{\mu}(g \circ f) = \int_{\Omega} (g \circ f) \ d \mu \leq M \int_{\Omega'}g \ d \mu' = M \varphi_{\mu'}(g).
\]
\end{proof}

\begin{proposition}\label{norm_L}
A measurable map $f : (\Omega,\cF,\mu) \to (\Omega', \cF',\mu')$ on measure spaces is bounded if and only if $L^{\infty}f$ is bounded. In that case, $|f|=|L^{\infty}f|$.
\end{proposition}
\begin{proof}
The proof of Proposition \ref{L} implies that $L^{\infty}f$ is bounded if $f$ is bounded, and $|L^{\infty}f| \leq |f|$.
Conversely, if $L^{\infty}f$ is bounded, then the characteristic function $\chi_{A}$ for $A \in \cF'$ induces the following inequality:
\[
\mu(f^{-1}(A))=\varphi_{\mu}(\chi_{f^{-1}(A)}) = \varphi_{\mu}(\chi_{A} \circ f) \leq |L^{\infty}f| \varphi_{\mu'}(\chi_{A}) =|L^{\infty}f|\mu'(A).
\]
This implies that $f$ is bounded, and $|f| \leq |L^{\infty}f|$.
\end{proof}

Next, we focus on the case of $C^{*}$-measure spaces associated with Borel measure spaces.
Let $\BMS$ denote the category of regular Borel measure spaces on compact Hausdorff spaces with bounded continuous maps, as a subcategory of $\CMS$. Similarly to the case of $L^{\infty}(-)$, the continuous function space $C(-)$ gives rise to a functor from $\BMS$ to $\QMS^{\op}$. It sends $(X,\cB(X),\mu)$ to $(C(X),\varphi_{\mu})$ and a bounded continuous map $f$ to $C(f)$ given by composition with $f$.

\begin{proposition}\label{C_functor}
$C : \BMS \to \QMS^{\op}$ is a functor.
\end{proposition}
\begin{proof}
The proof is similar to that of Proposition \ref{L}.
\end{proof}

\begin{proposition}\label{norm_C}
A continuous map $f : (X,\cB(X),\mu_{X}) \to (Y, \cB(Y),\mu_{Y})$ between Borel measure spaces is bounded if and only if $Cf$ is bounded. In that case, $|f|=|Cf|$.
\end{proposition}
\begin{proof}
If $f$ is bounded, then $Cf$ is bounded and $|Cf| \leq |f|$ by Proposition \ref{C_functor}.
The converse inequality is slightly different from that in the case of Proposition \ref{norm_L} because the characteristic map $\chi(A)$ is not continuous for $A \in \cB(Y)$ in general.
However, we can take a sequence of positive continuous functions $\{g_{n}\}$ on $Y$ converging to $\chi_{A}$ in $L^{2}(Y)$. If $Cf$ is bounded, then the inequality $\varphi_{\mu_{X}}(g_{n} \circ f) \leq |Cf| \varphi_{\mu_{Y}}(g_{n})$ for each $n$ induces
$\mu_{X}(f^{-1}(A)) \leq |Cf| \mu_{Y}(A)$ by $n \to \infty$. This implies that $f$ is bounded, and $|f| \leq |Cf|$.
\end{proof}

The Gelfand duality theorem involves the functor $C$ as an equivalence of categories between the category of compact Hausdorff spaces and the category of commutative $C^{*}$-algebras. Let us extend it to an equivalence between $\BMS$ and the category of commutative $C^*$-measure spaces, denoted by $\CCMS$, as a full subcategory of $\QMS$.

\begin{theorem}\label{main1}
The functor $C : \BMS \to \CCMS^{\op}$ is an equivalence of categories.
\end{theorem}
\begin{proof}
It suffices to show the essential surjectivity and fully faithfulness of $C$ by Theorem 1 of Section 4.4 in \cite{Mac98}.
Theorem \ref{RMK} states that $C$ is essentially surjective. The faithfulness of $C$ follows immediately from the Gelfand duality theorem by ignoring measures. Moreover, for a bounded morphism $f : C(X) \to C(Y)$ in $\CCMS^{\op}$, there exists a continuous map $g : X \to Y$ between compact Hausdorff spaces such that $Cg=f$.
By Proposition \ref{norm_C}, we have $|g| = |Cg|=|f|<\infty$.
Hence, we can conclude that $g$ is a morphism in $\BMS$ and confirm the fullness of $C$.
\end{proof}

The functor $C$ can be restricted to $C : \BMS_{\sharp} \to \CCMS^{\op}_{\sharp}$, where ${}_{\sharp}$ denotes the subcategory consisting of the same objects and measure-preserving morphisms.
We can show that a measurable map $f$ between measure spaces preserves measure if and only if $Cf$ does by an argument similar to that in the proof of Proposition \ref{norm_C}. In addition, Theorem \ref{RMK} involves the essential surjectivity of the restricted functor $\BMS_{\sharp} \to \CCMS^{\op}_{\sharp}$ of $C$. Hence, we can obtain the following corollary.

\begin{corollary}
The functor $C$ induces an equivalence of categories between $\BMS_{\sharp}$ and $\CCMS^{\op}_{\sharp}$.
\end{corollary}

We can also restrict the functor $C$ to probability spaces.
Let $\BPS$ (resp. $\BPS_{\sharp}$) denote the full subcategory of $\BMS$ (resp. $\BMS_{\sharp}$) consisting of probability spaces, and let $\CCPS$ (resp. $\CCPS_{\sharp}$) denote the full subcategory of $\CMS$ (resp. $\CCMS_{\sharp}$) consisting of quantum probability spaces.

\begin{corollary}
The functor $C$ induces an equivalence of categories between $\BPS$ (resp. $\BPS_{\sharp}$) and $\CCPS^{\op}$ (resp. $\CCPS_{\sharp}^{\op}$).
\end{corollary}

\section{Quantum conditional measure}

A classical conditional measure is essentially based on subspaces and restriction of a measure space. By focusing on the duality between classical and quantum measure spaces, as we have seen in the last part of Section 2, considering ideals and quotients of an algebra is a natural way to formulate quantum conditional measures. Accordingly, for a quantum measure space $(A,\varphi)$ and a two-sided ideal $I$ (simply referred to as ``ideal'' throughout this paper) of $A$, we aim to construct a quantum measure $\varphi_I$ on the quotient algebra $A/I$.

To build such a positive linear map, we will use orthogonal decomposition of Hilbert spaces. First, let us recall the GNS construction, which is a technique for establishing a Hilbert space from a quantum measure space \cite{Arv76}, \cite{KR97}.

\begin{definition}\label{GNS}
For a quantum measure space $(A,\varphi)$, let $N_{\varphi}$ denote the left ideal of $A$ given by $\{a \in A \mid \varphi(a^{\ast}a)=0\}$. The quotient vector space $A/N_{\varphi}$  admits an inner product $\i<[a]_{\varphi},[b]_{\varphi}>=\varphi(a^{*}b)$.
The Hilbert space $H_{\varphi}$ is defined as the completion of $A/N_{\varphi}$ with respect to the above inner product.
Multiplication on $A$ induces an algebra map $\pi : A \to B(H_{\varphi})$ such that $\pi(a)[b]_{\varphi}=[ab]_{\varphi}$ for $a \in A$, $[b]_{\varphi} \in A/N_{\varphi}$.
The {\em cyclic vector} $\xi \in H_{\varphi}$ is defined as $[e]_{\varphi}$ for the unit $e$ of $A$.
The original quantum measure $\varphi$ on $A$ can be expressed as
$\varphi(a)= w_{\xi}(\pi(a))=\i<\xi,\pi(a)\xi>$ by using the vector measure on $B(H_{\varphi})$ in Example \ref{vector_measure}.
The triple $(H_{\varphi},\pi,\xi)$ is called the {\em GNS construction} associated with $(A,\varphi)$.
\end{definition}

\begin{construction} For a quantum measure space $(A,\varphi)$, and an ideal $I$ on $A$,
let us construct a positive linear map $\varphi_{I} : A/I \to \bC$ as follows.
Suppose that $(H_{\varphi},\pi,\xi)$ is the GNS construction associated with $(A,\varphi)$ in Definition \ref{GNS}.
Consider the composition of the canonical projection and the inclusion to the completion
\[
(-)_{\varphi} : A \to A/N_{\varphi} \hookrightarrow H_{\varphi}.
\]
Let $I_{\varphi} \subset H_{\varphi}$ denote the closure of the image of $I$ by the above map.
It is equipped with the orthogonal decomposition $H_{\varphi} = I_{\varphi} \oplus I_{\varphi}^{\perp}$.
We express the decomposition of a vector $x \in H_{\varphi}$ as $x_{I} + x_{I}^{\perp} \in I_{\varphi} \oplus I_{\varphi}^{\perp}$.
Define $\varphi_{I} : A/I \to \bC$ by $\varphi_{I}[a]=\i<\xi_{I}^{\perp},a_{\varphi}>$. This map is well defined, i.e., it does not depend on the choice of the representative element, since $a_{\varphi}=(a_{\varphi})_{I}$ if $a \in I$ and $\i<\xi_{I}^{\perp},(a_{\varphi})_{I}>=0$.
Moreover, it is positive by the following calculation:
\begin{equation}
\begin{split}
\varphi_{I}([a]^{*}[a]) &= \i<\xi_{I}^{\perp},(a^{\ast}a)_{\varphi}> \\ \notag
&=\i<\xi_{I}^{\perp},\pi(a^{*})a_{\varphi}> \\
&=\i<\pi(a)\xi_{I}^{\perp},a_{\varphi}> \\
&=\i<(a_{\varphi})^{\perp}_{I}, (a_{\varphi})_{I}+(a_{\varphi})_{I}^{\perp}> \\
&=\i<(a_{\varphi})^{\perp}_{I},(a_{\varphi})^{\perp}_{I}> \geq 0.
\end{split}
\end{equation}
\end{construction}

We call $\varphi_I$ the {\em quantum conditional measure} on $A/I$ induced from $\varphi$. It is equipped with the canonical projection $p : (A,\varphi) \to (A/I,\varphi_I)$.
This is not measure-preserving, but bounded with norm $1$ in $\QMS$.

\begin{example}
Let $(\Omega,\cF,\mu)$ be a measure space. The GNS construction associated with $(L^{\infty}(\Omega),\varphi_{\mu})$ designates the Hilbert space $L^2(\Omega)$.
Given a subspace $B \in \cF$, consider the conditional measure space $(B,\cF_{B},\mu_{B})$.
The inclusion $i : B \hookrightarrow \Omega$ induces a surjective homomorphism $i^{*} : L^{\infty}(\Omega) \to L^{\infty}(B)$ given by the restriction of functions.
We obtain an ideal $I$ of $L^{\infty}(\Omega)$ as the kernel $\mathrm{Ker} i^{*} \cong L^{\infty}(B^{c})$.
Further, $i^*$ induces an isomorphism $L^{\infty}(\Omega)/I \to L^{\infty}(B)$ by the homomorphism theorem. This can be extended to an isomorphism $(L^{\infty}(\Omega)/I, (\varphi_{\mu})_I)  \to (L^{\infty}(B), \varphi_{\mu_B})$ in $\QMS_{\sharp}$.

If $\mu=P$ is a probability measure and the ideal $I=L^{\infty}(B^c)$ for a subspace $B \in \cF$ with $P(B) \neq 0$, then the normalization of the quantum conditional measure coincides with the classical conditional expectation \cite{Rao05}:
\[
\cN((\varphi_P)_I)[f]=\dfrac{(\varphi_P)_I[f]}{(\varphi_P)_I[\chi(\Omega)]}=  \dfrac{1}{P(B)} \int_{B} (f_{|B}(w)) \ d P_B(w)=
\int_{\Omega} f(w) \ d P(w|B)  = E_B(f).
\]
\end{example}

\begin{example}\label{C}
Let $(X,\cB(X),\mu)$ be a Borel measure space on a compact Hausdorff space $X$.
The GNS construction associated with $(C(X),\varphi_{\mu})$ designates the Hilbert space $L^2(X)$.
Given a closed subspace $B \in \cB(X)$, consider
the conditional measure space $(B,\cB(B),\mu_{B})$ on $B$.
Note that the induced homomorphism $i^{*} : C(X) \to C(B)$ from the inclusion $i : B \hookrightarrow X$ is not surjective in general. Let $I$ denote the kernel of $i^*$, which is an ideal of $C(X)$.
We have the quantum conditional measures $(\varphi_{\mu})_{I}$ on $C(X)/I$ and $\varphi_{\mu_B}$ on $C(B)$. Further, $i^*$ induces the following injection:
\[
C(X)/I \cong \mathrm{Im}i^* \hookrightarrow C(B).
\]
This is a measure-preserving homomorphism $(C(X)/I, (\varphi_{\mu})_I) \to (C(B), \varphi_{\mu_B})$ in $\QMS_{\sharp}$.
\end{example}

Conversely, every quantum conditional measure on a commutative $C^*$-measure space is essentially derived from classical conditional measures. To formulate it categorically, fix a $C^*$-measure space $(A,\varphi)$ and the associated Borel measure space $(X,\cB(X),\mu)$.
For a closed subspace $B$ in $X$, the inclusion induces an algebra homomorphism $C(X) \to C(B)$.
Let $I_B$ denote the kernel of this map, which is a closed ideal of $C(X)$.

\begin{theorem}\label{main2}
For any closed ideal $I$ of $A$, there exists a closed subspace $B$ in $X$ such that
$(A/I, \varphi_I) \cong (C(B), \varphi_{\mu_B})$ in $\QMS_{\sharp}$.
\end{theorem}
\begin{proof}
Let $B$ denote the compact Hausdorff space associated with the $C^*$-algebra $A/I$ via Gelfand duality.
The equivalence of categories $C$ assigns a continuous map $j : B \to X$ such that $j^*=C(f) : C(X) \to C(B)$ corresponds to the projection $A \to A/I$. In particular, $j^*$ is a surjection.
To show the injectivity of $j$, suppose that $j(a)=j(b)$ for $a, b \in B$. If $a \neq b$, we can choose a continuous function $f$ on $B$ satisfying $f(a)=1$ and $f(b)=0$ by Urysohn's lemma.
Since $j^*$ is a surjection, there exists $\tilde{f} : X \to \bC$ such that $\tilde{f} \circ j =f$.
However, $f(a) = \tilde{f}(j(a)) =\tilde{f}(j(b)) =f(b)$ contradicts the choice of function $f$.
Hence, $j$ is injective and $B$ can be regarded as a closed subspace of $X$.
\[
\xymatrix{
0 \ar[r] & I \ar[r] \ar[d]^{\cong} & A \ar[r] \ar[d]^{\cong}_{\alpha} & A/I \ar[r]  \ar@{..>}[d]^{\cong}_{\tilde{\alpha}}  & 0 \\
0 \ar[r] & I_B \ar[r]  & C(X) \ar[r]  & C(X)/I_B  \ar[r] & 0. \\
}
\]
In the above commutative diagram, $\tilde{\alpha}$ preserves measures with respect to $\varphi_I$ and $(\varphi_{\mu})_{I_B}$ since $\alpha$ preserves measures with respect to $\varphi$ and $\varphi_{\mu}$.
From Example \ref{C}, we have
\[
(A/I, \varphi_I) \cong (C(X)/I_B, (\varphi_{\mu})_{I_B}) \cong (C(B),\varphi_{\mu_{B}})
\]
in $\QMS_{\sharp}$.
\end{proof}

We have seen some commutative cases of quantum conditional measures derived from classical conditional measures. On the other hand, the non-commutative case is quite different from the above commutative cases.

\begin{remark}\label{simple}
A simple algebra does not have any proper two-sided ideal.
Hence, a quantum measure on a simple algebra has no (non-trivial) conditional measure.
For example, the trace measure on the $n$-th matrix algebra $M_{n}(\bC)$ has no quantum conditional measure.
\end{remark}

\begin{example}
Let $\cB(H)$ be the algebra of bounded linear operators on a separable infinite-dimensional Hilbert space $H$ with the vector measure $w_{h} : \cB(H) \to \bC$ for $h \in H$.
The subset $\cK(H)$ of compact operators on $H$ forms an ideal of $\cB(H)$. The quotient algebra
$\cC(H)=\cB(H)/\cK(H)$ is called the Calkin algebra. We have the quantum conditional measure $(w_{h})_{\cK(H)}$ on $\cC(H)$. The Calkin algebra is simple; hence, Remark \ref{simple} implies that we cannot update $(w_{h})_{\cK(H)}$ anymore.
\end{example}

\begin{remark}
The classical Bayes' rule relates a probability $P$ and the conditional probability $P(-|B)$ for a subspace $B$ with $P(B) \neq 0$. Here, we can describe the conditional probability as the normalization of the conditional measure. Let us consider this situation in quantum probability spaces.

Given a quantum probability space $(A,\varphi)$ and a proper ideal $I$ of $A$,
let $\varphi(-|I)$ denote the normalized state on $A/I$ for the quantum conditional measure $\varphi_I$. Let us express the ratio of $\varphi_I$ and $\varphi$ as $\varphi(I|a)=\varphi_I[a]/\varphi(a)$ for $\varphi(a) \neq 0$. Then, for $a \in A$, we have the formula
\[
\varphi([a]|I) = \dfrac{\varphi(I|a)}{\varphi_I[e]}\varphi(a)
\]
as an analog of Bayes' rule. Obviously, when $A=L^{\infty}(\Omega)$ and $I=L^{\infty}(B^c)$ for some classical probability space $(\Omega,\cF,P)$ with a subspace $B \in \cF$, the above equality represents the classical Bayes' rule by applying it to the characteristic function $\chi(A)$:
\[
P(A|B) = \dfrac{P(B|A)}{P(B)}P(A).
\]
\end{remark}


\begin{thebibliography}{AAA99}
\bibitem[AO03]{AO03}L. Accardi and N. Obata. \textit{Foundation of Quantum Probability} (Japanese). Makino Publ. (2003). v+295 pp.
\bibitem[Arv76]{Arv76} W. Arveson. An invitation to $C^{*}$-algebras. \textit{Graduate Texts in Mathematics}, No. 39. Springer-Verlag, New York-Heidelberg (1976). x+106 pp.
\bibitem[Bil95]{Bil95} P. Billingsley. \textit{Probability and measure}. 3rd edition. Wiley Series in Probability and Mathematical Statistics. A Wiley-Interscience Publication. John Wiley \& Sons, Inc., New York (1995). xiv+593 pp.
\bibitem[Gir82]{Gir82}M. Giry. A categorical approach to probability theory.  Categorical aspects of topology and analysis (Ottawa, Ont., 1980),  pp. 68--85, \textit{Lecture Notes in Math.}, 915, Springer, Berlin-New York (1982). \textit{Contemp. Math}., 167, $C\sp \ast$-algebras: 1943-1993 (San Antonio, TX, 1993),  2--19, Amer. Math. Soc., Providence, RI (1994).
\bibitem[GN94]{GN94}I. Gelfand and M. Neumark. On the imbedding of normed rings into the ring of operators in Hilbert space. Corrected reprint of the 1943 original [MR 5, 147].
\bibitem[HO07]{HO07}A. Hora and N. Obata. \textit{Quantum probability and spectral analysis of graphs}. With a foreword by Luigi Accardi. Theoretical and Mathematical Physics. Springer, Berlin (2007). xviii+371 pp.
\bibitem[IH92]{IH92}Y. Ito and T. Hamachi. \textit{Ergodic theorem and von Neumann algebra} (Japanese). Kinokuniya. Tokyo (1992). xii+413 pp.
\bibitem[KR97]{KR97}R. V. Kadison and J. R.  Ringrose. \textit{Fundamentals of the theory of operator algebras. Vol. I. Elementary theory}. Reprint of the 1983 original. Graduate Studies in Mathematics, 15. American Mathematical Society, Providence, RI (1997). xvi+398 pp.
\bibitem[Lyn]{Lyn}M. Lynn. Categories of probability spaces. \url{http://www.math.uchicago.edu/~may/VIGRE/VIGRE2010/REUPapers/Lynn.pdf}.
\bibitem[Mac98]{Mac98} S. Mac Lane. \textit{Categories for the working mathematician}. 2nd edition. Graduate Texts in Mathematics, 5. Springer-Verlag, New York (1998). xii+314 pp.
\bibitem[Rao05]{Rao05} M. M. Rao. \textit{Conditional measures and applications}. 2nd edition. Pure and Applied Mathematics (Boca Raton), 271. Chapman \& Hall/CRC, Boca Raton, FL (2005). xxiv+483 pp.
\bibitem[Rud87]{Rud87} W. Rudin. \textit{Real and complex analysis}. 3rd edition. McGraw-Hill Book Co., New York (1987). xiv+416 pp.
\end{thebibliography}
\end{document}